\documentclass{amsart}  
\usepackage{amssymb}  
\usepackage{hyperref}
\usepackage{amsthm}
\usepackage{amsmath}
\usepackage{mathtools}
\usepackage{graphicx}
\usepackage{enumerate}
\usepackage{color}
\usepackage{pinlabel}
\usepackage{rotating}
\usepackage{soul}

\newcommand{\bz}{\mathbb Z}

\def\fdtc{{\omega}}

\def\epsilon{\varepsilon}

\def\Q{\mathbb{Q}}

\def\Z{\mathbb{Z}}

\newtheorem{Thm}{Theorem}
\newtheorem{Prop}[Thm]{Proposition}

\newtheorem{Que}[Thm]{Question}

\theoremstyle{definition}

\theoremstyle{remark}
\newtheorem{Rem}[Thm]{Remark}

\theoremstyle{definition}

\author{Peter Feller}
\address{ETH Z\"urich, R\"amistrasse 101, 8092 Z\"urich, Switzerland}
\email{peter.feller@math.ch}
\author{Diana Hubbard}
\address{Brooklyn College, 2900 Bedford Avenue
Brooklyn, NY 11210-2889, United States}
\email{diana.hubbard@brooklyn.cuny.edu}
\makeatletter
\@namedef{subjclassname@2020}{%
  \textup{2020} Mathematics Subject Classification}
\makeatother

\begin{document}  

\title[Examples of non-minimal 
open books with high FDTC]{Examples of non-minimal 
open books with high fractional Dehn twist coefficient}

\begin{abstract}
In this short note we construct examples of open books for 3-manifolds that show that arbitrarily high twisting of the monodromy of the open book does not guarantee maximality of the Euler characteristic of the pages among the open books supporting the same contact manifold. We find our examples of open books 
as the double branched covers of families of closed
braids studied by Malyutin and Netsvetaev. 
\end{abstract}
\maketitle
\section{Introduction}

 Denote by $S_{g,b}$ the compact oriented connected surface of genus $g\geq 0$ with $b\geq 1$ boundary components.
Denote by $\mathrm{MCG}(S_{g,b})$ the mapping class group of $S_{g,b}$
 .
 We recall that every conjugacy class $[\phi]$ of a $\phi\in \mathrm{MCG}(S_{g,b})$ determines a unique closed oriented connected $3$-manifold $M_\phi$ together with an open book $O_\phi$ on $M_\phi$ that has $S_{g,b}$ as its pages and $[\phi]$ as its monodromy. We denote by $\fdtc(\phi) \in\Q$ the fractional Dehn twist coefficient of $\phi\in \mathrm{MCG}(S_{g,1})$. Roughly, $\fdtc(\phi)$ measures the amount of twisting effected about the boundary of $S_{g,1}$ by $\phi$.

\begin{Que}\label{q} Fix an integer $g\geq 1$. Let $\phi$ in $\mathrm{MCG}(S_{g,1})$ satisfy $|\fdtc(\phi)|>g-\tfrac{1}{2}$.
\begin{itemize}

\item Do the pages of open books for $M_\phi$ have Euler characteristic at most~$1-2g$?
\item Do the pages of open books for $M_\phi$ that support the same contact structure as $O_\phi$ have Euler characteristic at most~$1-2g$?
\end{itemize}
\end{Que}
As a variant of Question~\ref{q}, we ask: does there exist a constant $c_g$ such that there are positive answers to the above questions when replacing the assumption $|\fdtc(\phi)|>g-\tfrac{1}{2}$ with $|\fdtc(\phi)|>c_g$?

Motivated by the fact that $|\fdtc(\phi)|>1$ implies that the the open book $O_\phi$ on $M_\phi$ cannot be destabilized (destabilization is an operation that increases the Euler characteristic of pages without changing the manifold), a first naive hope could be the following: the above questions can be answered in the positive for $c_g=1$. In other words, maybe, if $|\fdtc(\phi)|>1$, then the pages of the open book $O_\phi$ have the largest Euler characteristic among the pages among the open books for $M_\phi$, or at least among the open books that support the same contact structure as $M_\phi$.
However, this is wrong:

\begin{Thm}\label{thm:FDTCandgenus} For every 
integer $k\geq 0$, there exists a Stein fillable contact $3$-manifold 
supported by an open book with connected binding and pages of genus $k+1$ whose monodromy has fractional Dehn twist equal to $k$ such that its contact structure is also supported by an open book with connected binding and pages of genus $k$.
\end{Thm}

In fact, Theorem~\ref{thm:FDTCandgenus} implies that any constants $c_g$ as above have to grow at least linearly in $g$. We point to the end of Section~\ref{sec:ContMot} for (circumstantial)  evidence that such constants $c_g$ might exist and can be chosen to grow at most linearly in $g$; in other words, evidence that Question~\ref{q} could have a positive answer.

Finally, we briefly compare Question \ref{q} to an open question in contact geometry. Every overtwisted contact structure can be represented by a planar open book decomposition, but the same is not true for tight contact structures: there exist such structures that require pages of genus one  \cite{etnyre2004planar}. It is an open question whether there exist any tight contact structures that cannot be represented by open book decompositions whose pages have genus zero or one. By \cite{honda_kazez_matic_II} any contact manifold that can be represented by an open book decomposition $O$ (with connected binding) whose fractional Dehn twist coefficient is greater than or equal to one must carry a tight contact structure, namely the one supported by $O$. A positive answer to Question \ref{q}, when restricted just to positive fractional Dehn twist coefficients, would provide information about what types of open books for such tight contact structures are possible. However, it would not directly address the open question.

\subsection*{Organization}
In Section~\ref{sec:ContMot} we discuss the setup and motivation,
in Section~\ref{sec:Constr} we find the open books in Theorem~\ref{thm:FDTCandgenus} as the double branched covers of families of closed braids studied by Malyutin and Netsvetaev~\cite{malyutin_netsvetaev}, and in Section~\ref{sec:distinct} we discuss the distinctness of the underlying $3$-manifolds of the open books we construct.
\section{Context and motivation: from braided links to open books via double branched covering}\label{sec:ContMot}
As a general reference for open books and contact structures we point to~\cite{etnyre2006lectures}. For a general reference for braids we point to~\cite{Birman_74_BraidsLinksAndMCGs, birman2005braids}.
\subsection*{Braids, braided links, and transverse links} Links are isotopy classes of smooth nonempty closed oriented $1$-manifolds in $S^3$.
A \emph{braiding} of a link $L$ is a choice of representative of $L$ that is transverse to the trivial open book of $S^3$ (the one with binding the unknot $U$).
We consider braidings up to isotopy in the complement of $U$ transverse to all pages.
We refer to a link with a choice of braiding as a \emph{braided link}
. The number of transverse intersections of a braided link with a page is called \emph{the number of strands}. 
The minimal number of strands among braidings of a link $L$ is called the \emph{braid index} of $L$. Braided links with $n\geq 1$ strands are canonically identified with conjugacy classes in \emph{Artin's braid group} on $n$-strands $B_n$~\cite{artin_TheorieDerZoepfe}. Denoting by $D_n$ the closed disc with $n$ punctures, $B_n$ can be defined as the mapping class group $\mathrm{MCG}(D_n)$.
We write $\widehat{\beta}$ for the braided link given by the conjugacy class of a \emph{braid} $\beta$. Any two braidings of a given link are related by a sequence of so-called Markov stabilizations and destabilizations.

A braided link canonically determines a \emph{transverse link}---a transverse isotopy class of a link transverse to the standard contact structure on $S^3$ (the one corresponding to the trivial open book). In fact, two braided links determine the same transverse link if and only if they are related by positive Markov stabilizations and destabilizations~\cite{wrinkle2002markov, OrevShev}.

\subsection*{Double branched covers}
The double branched cover construction associates to a link $L$ an oriented closed connected $3$-manifold $\Sigma(L)$. This extends to braided links and open books:
\[\Sigma\colon\{\text{braided links}\}\to \{\text{open books on closed connected oriented $3$-manifolds}\}.\] Indeed, denoting by $\pi: \Sigma(L)\to S^3$ the double branching where a braiding of $L$ is fixed,
every page $P$ of the trivial open book on $S^3$ has $\pi^{-1}(P)$ as the pages of an open book:
$\pi$ restricts to a double branched covering $\pi^{-1}(P)\to P$ of $P$ along the intersection of $P$ with the braiding. In particular, a braided link with $n$-strands yield an open book with pages $\Sigma(D_n)$, where $\Sigma(D_n)$ is $S_{\frac{n-1}{2},1}$ respectively $S_{\frac{n-2}{2},2}$ for $n$ odd and even, respectively.

For another perspective, we note that 
$\Sigma$ is induced by the Birman-Hilden embedding of groups
\[\mathrm{BH}\colon B_n\hookrightarrow \mathrm{MCG}(\Sigma(D_n)),\]
by considering the induced map that maps a conjugacy class $\widehat{\beta}$ of a braid $\beta$ to the conjugacy class $[\mathrm{BH}(\beta)]$ of $\mathrm{BH}(\beta)$ and considering the corresponding open book~$O_{\mathrm{BH}(\beta)}$. This follows from the fact that $\mathrm{BH}$ is defined by taking the lift to the double branched cover; see e.g.~\cite[Chapter~5]{farb2011primer}.

Recall that from an open book on a closed oriented $3$-manifold one can construct others by so-called stabilizations and destabilizations (also known as Hopf plumbing and Hopf deplumbing).
If two open books are related by positive stabilization and destabilization then they support the same contact structure; and conversely, any two open books that support the same contact structure are related by a sequence of so-called \emph{positive} stabilizations and destabilizations by a result known as Giroux's correspondence~\cite{giroux2003g}.
With this setup we note the following.

If a braided link $\widehat{\beta'}$ is obtained from a braided link $\widehat{\beta}$ by a positive (negative) Markov stabilization, then the open book $\Sigma\left(\widehat{\beta'}\right)$ is obtained from  $\Sigma\left(\widehat{\beta}\right)$ by a positive (negative) stabilization. In particular, a transverse link (i.e.~a class of braided links related by positive Markov stabilizations and destabilizations) give rise to a 
class of open books up to positive stabilizations and destabilizations, and hence a contact structure on the double branched cover of the link; see e.g.~\cite{plamenevskaya_doublebranched}.

\subsection*{Fractional Dehn twist coefficient}
For a compact connected surface of finite type with boundary $S_{g,b,p}$ (genus $g\geq0$, $b\geq1$ boundary components, $p\geq 0$ punctures), one considers the so-called fractional Dehn twist coefficient, a homogeneous quasimorphism
$\fdtc\colon \mathrm{MCG}(S_{g,b,p})\to \Q$ with respect to a fixed boundary component; see~\cite{honda2007right, honda_kazez_matic_II} and compare also, for instance, to~\cite{hedden_mark, ItoKawamuro_18_FDTC} .
To avoid dependence on the choice of boundary, we restrict to the case $b=1$.

Under the Birman-Hilden map, the fractional Dehn twist coefficient behaves simply:
for all odd $n\geq 1$,
\begin{equation}\label{eq:BHandFDTC}\fdtc(\mathrm{BH}(\beta))=\frac{\fdtc(\beta)}{2}\quad \text{for all}\quad\beta\in B_n;\end{equation}
see~\cite[Theorem 4.2]{ItoKawamuro_18_FDTC}. (The assumption that $n$ is odd is to assure that $\Sigma(D_n)$ has one boundary component.) A key observation to see~\eqref{eq:BHandFDTC} is the following. For the full-twist $\Delta^2=\delta^n\in B_n$---the positive Dehn twist along the boundary of $D_n$---we have that $\mathrm{BH}((\Delta^2)^2)$ is the Dehn twist along the boundary of $\Sigma(D_n)$. 

\subsection*{Analogue of Question~\ref{q} for braidings of links}

Before we use the above setup to discuss our examples in Section~\ref{sec:Constr}, we discuss why we dare to hope that Question~\ref{q} has a positive answer.
Considering the double branched cover construction, the following is an analogue of Question~\ref{q}.

\begin{Que}
\label{q'}
Fix an integer $n\geq 3$.
Let $\beta$ in $B_n$ satisfy $|\fdtc(\beta)|>n-2$ and denote by $L$ and $L^\mathrm{trans}$ the link isotopy class and transverse link isotopy class obtained as the closure of $\beta$, respectively.
\begin{itemize}
\item Does every braid with closure $L$ have at least $n$ strands? \cite[Quest.~7.3]{feller_hubbard}
\item Does every braid with transverse closure $L^\mathrm{trans}$ have at least $n$ strands?
\end{itemize}
\end{Que}
And indeed, taking $n$ odd, and noting that
\[\fdtc(\mathrm{BH}(\beta))\overset{\text{\eqref{eq:BHandFDTC}}}{=}\frac{\fdtc(\beta)}{2}\quad\text{and}\quad g\coloneqq\textrm{genus}(\Sigma(D_n))=\tfrac{n-1}{2},\]
we see that the assumptions of the two questions correspond:
\[|\fdtc(\beta)|>n-2\quad\text{if and only if}\quad|\fdtc(\mathrm{BH}(\beta))|>g-\tfrac{1}{2}.\]

We were able to answer Question~\ref{q'} in the positive for $n=3$ and,  for general $n$, with the stronger assumption $|\fdtc(\beta)|>n-1$; see~\cite{feller_hubbard}. Hence, we feel justified to ask Question~\ref{q}.

\section{Construction of the examples}\label{sec:Constr}

We describe our examples for the proof of Theorem~\ref{thm:FDTCandgenus} as double branched covers of braided links. To describe braids, we use the standard Artin generators $\sigma_1$, $\cdots$, $\sigma_{n-1}$ for the braid group on $n$-strands~\cite{artin_TheorieDerZoepfe}.
 
For any pair of integers $n, m\geq 1$, define the braid $\beta_{n,m}\coloneqq (\delta \delta^{\Delta})^{m-1} \delta$, where
\[ \delta \coloneqq \sigma_{1}\sigma_{2} \cdots \sigma_{n-1}\quad\text{and}\quad\delta^{\Delta}  \coloneqq \sigma_{n-1}\sigma_{n-2} \cdots \sigma_{1}.\]
Malyutin and Netsvetaev observed that the closures of $\beta_{n,m}$ and $\beta_{m,n}$ are isotopic as links in $S^3$; see \cite[Figure~2]{malyutin_netsvetaev}. In fact, we can say more: Etnyre and Van Horn-Morris showed that any two positive braids representing the same link $L$ are related by positive Markov stabilizations and destabilizations and braid isotopy \cite[Corollary~1.13]{etnyre_vanhornmorris_transverse}. In particular, since both $\beta_{n,m}$ and $\beta_{m,n}$ are positive braids, we have the first two items of the following.

\begin{Prop}\label{prop:transverseisotopycontactomorphicfillable} Fix integers $n,m\geq 1$ and
denote by $L$ the closure of $\beta_{n,m}$ and $\beta_{m,n}$ as links in $S^3$. Set $O_{n,m}\coloneqq O_{\mathrm{BH}(\beta_{n,m})}=\Sigma\left(\widehat{\beta_{n,m}}\right)$ and $O_{m,n}\coloneqq\Sigma\left(\widehat{\beta_{m,n}}\right)$ the open books on $\Sigma(L)$ corresponding to $\beta_{n,m}$ and $\beta_{m,n}$, respectively.
We have that\begin{enumerate}[i)]
\item\label{i} the closures of $\beta_{n,m}$ and $\beta_{m,n}$ as transverse links in $S^3$ with the standard contact structure are transversely isotopic,
\item\label{ii} the two open books $O_{n,m}$ and $O_{m,n}$ correspond to the same contact structure $\xi_{\{n,m\}}$ on $\Sigma(L)$,
\item\label{iii}
the contact manifold $(\Sigma(L), \xi_{\{n,m\}})$ is Stein fillable (and therefore tight)
\item\label{iv} for $n$ odd, the fractional Dehn twist coefficients of $\mathrm{BH}(\beta_{n,m})$ in $\mathrm{MCG}(\Sigma(D_n))$---the monodromy of the open book $O_{n,m}$---is $\fdtc(\mathrm{BH}(\beta_{n,m}))=(m-1)/2$. 
\end{enumerate}
\end{Prop}

\begin{proof} \ref{i}) holds since $\beta_{n,m}$ and $\beta_{m,n}$ are related by positive Markov stabilizations and destabilizations.
\ref{ii}) holds since the corresponding open books $O_{n,m}$ and $O_{m,n}$ are related by corresponding positive stabilizations and destabilizations.

We discuss \ref{iii}).
It is a theorem of Giroux (\cite{giroux2003g}; see also Etnyre \cite[Theorem 5.11]{etnyre2006lectures}) that a contact manifold is Stein fillable if and only if there is an open book for it whose monodromy can be written as a composition of positive Dehn twists. The monodromy of $O_{n,m}$ is $[\mathrm{BH}(\beta_{n,m})]$, and $\mathrm{BH}(\beta_{n,m})$ is a composition of positive Dehn twists since $\beta_{n,m}$ is a composition of positive braid generators. 

For \ref{iv}), we use $\fdtc(\beta_{n,m})=m-1$ (see~\cite[Example~6.1]{feller_hubbard}) and~\eqref{eq:BHandFDTC}.
\end{proof}

\begin{proof}[Proof of Theorem~\ref{thm:FDTCandgenus}]

Fix an non-negative integer $k$. Consider the braid $\beta_{2k+3,2k+1}$ and denote by $K_k$ its closure as a knot in $S^3$. The fractional Dehn twist coefficient of the monodromy of the corresponding open book $O_{2k+3,2k+1}$ on $\Sigma(K_k)$ is $k$. The pages of $O_{2k+3,2k+1}$ have genus $k+1$. As the proposed contact manifold, we take the contact manifold $(\Sigma(K_k), \xi_{\{2k+3,2k+1\}})$ corresponding to the open book decomposition $O_{2k+3,2k+1}$.
By Proposition~\ref{prop:transverseisotopycontactomorphicfillable}, the contact manifold $(\Sigma(K_k), \xi_{\{2k+3,2k+1\}})$ 
is the same as to one corresponding to the open book $O_{2k+1,2k+3}$. We conclude the proof by noting that the pages of $O_{2k+1,2k+3}$ have genus $k$ and that the contact manifold $(\Sigma(K_k), \xi_{\{2k+3,2k+1\}})$ is Stein fillable by Proposition~\ref{prop:transverseisotopycontactomorphicfillable}.
\end{proof}

\section{Distinctness}\label{sec:distinct}
The examples of contact manifolds we gave in Theorem \ref{thm:FDTCandgenus} 
are pairwise distinct. In fact, the underlying manifolds are all pairwise non-homeomorphic:
\begin{Prop}\label{prop:orderofH1}
For every integer $k\geq 1$, denote by $K_k$ the closure of $\beta_{2k+1,2k+3}$ as a knot in $S^3$.
We have $|H_1(\Sigma(K_k);\bz)|=4k^2+4k-1$.
\end{Prop}
\begin{Rem}
Even without Proposition~\ref{prop:orderofH1}, it is clear that for every $k$, there exists an $l_0$ such that $\Sigma(K_l)$ is not homeomorphic to $\Sigma(K_k)$ for $l\geq l_0$. Indeed, for every closed oriented $3$-manifold $M$, there exists a constant $c_M$ such that all monodromies $\phi$ of open books with connected binding on $M$ satisfy $|\fdtc(\phi)|\leq c_M$; see~\cite{hedden_mark}.   
\end{Rem}
We establish Proposition~\ref{prop:orderofH1} using that $|H_1(\Sigma(K);\bz)|=|\mathrm{det}(K)|=|\Delta_K(1)|$ for all knots $K$, where $\mathrm{det}(K)$ and $\Delta(K)$ denote the knot determinant and the Alexander polynomial of $K$, respectively, 
and using the following connection to the Burau representation. For integers $n\geq 1$, we have
\[
|\Delta_K(1)|= \mathrm{det}\left(I_{n-1}-f_*(\beta)(1)\right)\text{ for all }\beta\in B_n,\]
where $I_{n-1}$ denotes the identity matrix in $GL_{n-1}(\Z)$ and $f_*\colon B_n\to GL_{n-1}(\Z)$ denotes the reduced Burau representation evaluated at $t=-1$~\cite{birman2005braids}.

\begin{proof}[Proof of Proposition~\ref{prop:orderofH1}]
For integers $i\geq 1, n\geq 1$, we set $\sigma_{i,n}\coloneqq f_*(\sigma_i)$ for $\sigma_i\in B_n$.
Using the explicit values of $f_*(\sigma_i)$ (compare~\cite{Birman_74_BraidsLinksAndMCGs}), one finds, for $n\geq3$:
\begin{eqnarray}
\label{eq:delta}f_*(\delta)=\sigma_{1,n}\sigma_{2,n}\cdots\sigma_{n-1,n}&=&\left[\begin{array}{c|c}
0\quad\cdots\quad 0&1\\
\hline
&1\\
-I_{n-2}&\vdots\\
&1\\
\end{array}\right] \text{ and }\\
\label{eq:deltaop}f_*(\delta^\Delta)=\sigma_{n-1,n}\sigma_{k-2,n}\cdots\sigma_{1,n}&=&\left[\begin{array}{c|c}
1&\\
-1&I_{n-2}\\
\vdots &\\
(-1)^{n-1}&\\
\hline
(-1)^{n}&0\quad\cdots\quad 0\\
\end{array}\right].
\end{eqnarray}
Indeed, \eqref{eq:delta} and~\eqref{eq:deltaop} are easily checked for $n=3$, and, for $n\geq 4$, \eqref{eq:delta} and~\eqref{eq:deltaop} follow inductively using that 
we have
\begin{eqnarray*}
\sigma_{1,n}\sigma_{2,n}\cdots\sigma_{n-2,n}&=&\left[\begin{array}{c|c}
&1\\
\sigma_{1,n-1}\sigma_{2,k-1}\cdots\sigma_{n-2,n-1}&\vdots\\
&1\\
\hline
0\quad\quad\quad\quad\cdots\quad\quad\quad\quad 0&1\\
\end{array}\right]\text{ and }\\
\sigma_{n-2,n}\sigma_{n-3,n}\cdots\sigma_{1,n}&=&\left[\begin{array}{c|c}
&1\\
\sigma_{n-2,n-1}\sigma_{n-3,n-1}\cdots\sigma_{1,n-1}&\vdots\\
&1\\
\hline
0\quad\quad\quad\quad\cdots\quad\quad\quad\quad 0&1\\
\end{array}\right].
\end{eqnarray*}

So, we see that, for $n\geq 3$ odd, we have
\[(f_*(\delta\delta^\Delta))^2=(f_*(\delta)f_*(\delta^\Delta))^2\overset{\text{\eqref{eq:delta},\eqref{eq:deltaop}}}{=}
\left[\begin{array}{c|c}
-1&0\quad\cdots\quad 0\\
\hline
-2&\\
0&\\
-2&-I_{n-2}\\
\vdots &\\
0&\\
-2&\\
\end{array}\right]^2=\left[\begin{array}{c|c}
1&0\quad\cdots\quad 0\\
\hline
4&\\
0&\\
4&I_{n-2}\\
\vdots &\\
0&\\
4&\\
\end{array}\right],\] and hence, for all integers $l\geq 1$, we have that $f_*(\delta\delta^\Delta)^{2l}f_*(\delta)$ equals
\[\left[\begin{array}{c|c}
1&0\cdots 0\\
\hline
4&\\
0&\\
4&I_{n-2}\\
\vdots &\\
0&\\
4&\\
\end{array}\right]^lf_*(\delta)=\left[\begin{array}{c|c}
1&0\cdots 0\\
\hline
4l&\\
0&\\
4l&I_{n-2}\\
\vdots &\\
0&\\
4l&\\
\end{array}\right]\left[\begin{array}{c|c}
0\cdots0&1\\
\hline
&1\\
&1\\
-I_{n-2}&1\\
&\vdots\\
&1\\
&1\\
\end{array}\right]=\left[\begin{array}{c|c}
0\cdots0&1\\
\hline
&4l+1\\
&1\\
-I_{n-2}&4l+1\\
&\vdots\\
&1\\
&4l+1\\
\end{array}\right].\]
We conclude that 
$\mathrm{det}\left(I_{2k}-f_*(\beta_{2k+3,2k+1})\right)
=\mathrm{det}\left(I_{2k}-f_*(\delta\delta^\Delta)^{2(k+1)}f_*(\delta)\right)$ equals
\begin{align*}
\mathrm{det}\left(
\left[\begin{array}{ccccc|c}
1&0&\cdots&&0&-1\\
\hline
1&1&0&\cdots&0&-4(k+1)-1\\
0&1&1&&&-1\\
&&\ddots&&&\vdots\\
&&&1&1&-1\\
&&&&1&-4(k+1)\\
\end{array}\right]
\right)
=-4k^2-4k+1,\end{align*}
where the last equality follows for example by developing the determinant using the last column.
\end{proof}

\newpage

\bibliographystyle{amsalpha}\bibliography{reference}

\providecommand{\bysame}{\leavevmode\hbox to3em{\hrulefill}\thinspace}
\providecommand{\MR}{\relax\ifhmode\unskip\space\fi MR }
\providecommand{\MRhref}[2]{%
  \href{http://www.ams.org/mathscinet-getitem?mr=#1}{#2}
}
\providecommand{\href}[2]{#2}
\begin{thebibliography}{EVHM11}

\bibitem[Art25]{artin_TheorieDerZoepfe}
Emil Artin, \emph{Theorie der {Z}{\"o}pfe}, Abh. Math. Sem. Univ. Hamburg
  \textbf{4} (1925), no.~1, 47--72. \MR{3069440}

\bibitem[BB05]{birman2005braids}
Joan~S Birman and Tara~E Brendle, \emph{Braids: a survey}, Handbook of knot
  theory, Elsevier, 2005, pp.~19--103.

\bibitem[Bir74]{Birman_74_BraidsLinksAndMCGs}
Joan~S. Birman, \emph{Braids, links, and mapping class groups}, Princeton
  University Press, Princeton, N.J., 1974, Annals of Mathematics Studies, No.
  82. \MR{0375281 (51 \#11477)}

\bibitem[Etn04]{etnyre2004planar}
John~B Etnyre, \emph{Planar open book decompositions and contact structures},
  International Mathematics Research Notices \textbf{2004} (2004), no.~79,
  4255--4267.

\bibitem[Etn06]{etnyre2006lectures}
\bysame, \emph{Lectures on open book decompositions and contact}, Floer
  Homology, Gauge Theory, and Low-dimensional Topology: Proceedings of the Clay
  Mathematics Institute 2004 Summer School, Alfr{\'e}d R{\'e}nyi Institute of
  Mathematics, Budapest, Hungary, June 5-26, 2004, vol.~5, American
  Mathematical Soc., 2006, p.~103.

\bibitem[EVHM11]{etnyre_vanhornmorris_transverse}
John~B. Etnyre and Jeremy Van Horn-Morris, \emph{Fibered transverse knots and
  the {B}ennequin bound}, Int. Math. Res. Not. IMRN (2011), no.~7, 1483--1509.
  \MR{2806512}

\bibitem[FH19]{feller_hubbard}
Peter Feller and Diana Hubbard, \emph{Braids with as many full twists as
  strands realize the braid index}, J. Topol. \textbf{12} (2019), no.~4,
  1069--1092, Arxiv:1708.04998 [math.GT]. \MR{3977871}

\bibitem[FM11]{farb2011primer}
Benson Farb and Dan Margalit, \emph{A primer on mapping class groups (pms-49)},
  Princeton University Press, 2011.

\bibitem[Gir03]{giroux2003g}
Emmanuel Giroux, \emph{G\'eom\'etrie de contact: de la dimension trois vers les
  dimensions sup\'erieures}, arXiv preprint math/0305129 (2003).

\bibitem[HKM07]{honda2007right}
Ko~Honda, William~H Kazez, and Gordana Mati{\'c}, \emph{Right-veering
  diffeomorphisms of compact surfaces with boundary}, Inventiones mathematicae
  \textbf{169} (2007), no.~2, 427--449.

\bibitem[HKM08]{honda_kazez_matic_II}
Ko~Honda, William~H. Kazez, and Gordana Mati\'c, \emph{Right-veering
  diffeomorphisms of compact surfaces with boundary. {II}}, Geom. Topol.
  \textbf{12} (2008), no.~4, 2057--2094. \MR{2431016}

\bibitem[HM18]{hedden_mark}
Matthew Hedden and Thomas~E. Mark, \emph{Floer homology and fractional {D}ehn
  twists}, Adv. Math. \textbf{324} (2018), 1--39. \MR{3733880}

\bibitem[IK18]{ItoKawamuro_18_FDTC}
Tetsuya Ito and Keiko Kawamuro, \emph{On the fractional dehn twist coefficients
  of branched coverings}, arxiv preprint (2018), Arxiv:1807.04398 [math.GT].

\bibitem[MN03]{malyutin_netsvetaev}
A.~V. Malyutin and N.~Yu. Netsvetaev, \emph{Dehornoy order in the braid group
  and transformations of closed braids}, Algebra i Analiz \textbf{15} (2003),
  no.~3, 170--187. \MR{2052167}

\bibitem[OS03]{OrevShev}
S.~Yu. Orevkov and V.~V. Shevchishin, \emph{Markov theorem for transversal
  links}, J. of Knot Theory and Ramifications \textbf{12} (2003), 905--913.

\bibitem[Pla06]{plamenevskaya_doublebranched}
Olga Plamenevskaya, \emph{Transverse knots, branched double covers and
  {H}eegaard {F}loer contact invariants}, J. Symplectic Geom. \textbf{4}
  (2006), no.~2, 149--170. \MR{2275002}

\bibitem[Wri02]{wrinkle2002markov}
Nancy~C Wrinkle, \emph{The {M}arkov theorem for transverse knots}, Ph.D.
  thesis, Columbia University, 2002.

\end{thebibliography}

\end{document}